\documentclass[fleq]{amsart}
\usepackage{amssymb,amsmath,latexsym,amsbsy,color,enumerate}
\numberwithin{equation}{section}

\newtheorem{theorem}{Theorem}

\newtheorem{lemma}{Lemma}

\newtheorem{condition}[theorem]{Condition}
\begin{document}
 \title[Automorphisms of blowups of threefolds]{Automorphisms of blowups of threefolds being Fano or having Picard number $1$}
 \author{Tuyen Trung Truong}
    \address{School of Mathematics, Korea Institute for Advanced Study, Seoul 130-722, Republic of Korea}
 \email{truong@kias.re.kr}
\thanks{}
    \date{\today}
    \begin{abstract}Let $X_0$ be a smooth projective threefold which is Fano or which has Picard number $1$. Let $\pi :X\rightarrow X_0$ be a finite composition of blowups along smooth centers. We show that for "almost all" of such $X$, if $f\in Aut(X)$ then its first and second dynamical degrees are the same. We also construct many examples of finite blowups $X\rightarrow X_0$, on which any automorphism is of zero entropy.  
    
The main idea is that because of the log-concavity of dynamical systems and the invariance of Chern classes under holomorphic automorphisms, there are some constraints on the nef cohomology classes. 

We will also discuss a possible application of these results to a threefold constructed by Kenji Ueno. 
\end{abstract}
\maketitle
\section{Introduction}
\label{SectionIntroduction}
While there are many examples of compact complex surfaces having automorphisms of positive entropies (works of Cantat \cite{cantat}, Bedford-Kim \cite{bedford-kim2}\cite{bedford-kim1}\cite{bedford-kim}, McMullen \cite{mcmullen3}\cite{mcmullen2}\cite{mcmullen1}\cite{mcmullen}, Oguiso \cite{oguiso3}\cite{oguiso2}, Cantat-Dolgachev \cite{cantat-dolgachev}, Zhang \cite{zhang}, Diller \cite{diller}, D\'eserti-Grivaux \cite{deserti-grivaux}, Reschke \cite{reschke},...), there are few interesting examples of manifolds of higher dimensions having automorphisms of positive entropies (Oguiso \cite{oguiso1}\cite{oguiso}, Oguiso-Perroni \cite{oguiso-perroni},...). In particular, for the class of smooth rational threefolds, there are currently only two known examples of manifolds with primitive automorphisms of positive entropy (see \cite{oguiso-truong, catanese-oguiso-truong, colliot-thelene}). Here  a primitive automorphism, defined by D.-Q. Zhang \cite{zhang}, is one that has no non-trivial invariant fibrations over a base of dimension $1$ or $2$. For general properties on automorphism groups of compact K\"ahler manifolds, see the recent survey paper \cite{dinh}.   
 
Then, it is natural to ask for what happens in dimension $3$ and higher. For example, the following question was asked by Eric Bedford in 2011: 

{\bf Question 1}. Is there a finite composition of blowups at points or smooth curves $X\rightarrow \mathbb{P}^3$ starting from $\mathbb{P}^3$ and an automorphism $f:X\rightarrow X$ with positive entropy?
 
This paper aims to study Question 1 and some related questions. We give many evidences to that the answer to Question 1 is negative and to that the examples in \cite{oguiso-truong, catanese-oguiso-truong, colliot-thelene} can not be obtained as smooth blowups of smooth threefolds having Picard number $1$ or being Fano.   

Our results and proofs are stated in terms of dynamical degrees, which we recall now. Let $X$ be a smooth projective threefold. We denote by $Pic (X)$ the Picard group of $X$, $Pic _{\mathbb{Q}}(X)=Pic (X)\otimes _{\mathbb{Z}}\mathbb{Q}$ and $Pic _{\mathbb{R}}(X)=Pic (X)\otimes _{\mathbb{Z}}\mathbb{R}$. Let $Nef(X)\subset Pic_{\mathbb{R}}(X)$ be the cone of nef classes, which is the closure of the cone of ample classes. By Kleiman's criterion, a class in $Pic _{\mathbb{R}}(X)$ is nef iff it has non-negative intersection with every curve on $X$. For later use, we denote by  $c_1(X)$ and $c_2(X)$ the first and second Chern classes of $X$. Let $f:X\rightarrow X$ be an automorphism. Then $f$ preserves both $Pic (X)$ and $Nef(X)$. Let $\omega$ be an ample class on $X$. We define the first and second dynamical degrees of $f$ as follows:
\begin{eqnarray*}
\lambda _1(f)&=&\lim _{n\rightarrow\infty}[(f^n)^*(\omega ).\omega ^2]^{1/n},\\
\lambda _2(f)&=&\lim _{n\rightarrow\infty}[(f^n)^*(\omega ^2).\omega ]^{1/n}.
\end{eqnarray*}
Here are some properties of these dynamical degrees: $\lambda _1(f)^2\geq \lambda _2(f)\geq 1$ and $\lambda _1(f^{-1})=\lambda _2(f)$. For more on dynamical degrees see \cite{dinh-sibony2}. 

Entropy of $f$ can be computed via dynamical degrees by Gromov-Yomdin's theorem \cite{gromov, yomdin}: $h_{top}(f)=\log \max \{\lambda _1(f),\lambda _2(f)\}$.  Hence, $f$ has positive entropy iff $\lambda _1(f)>1$. 

Primitivity of $f$ can also be detected from dynamical degrees via the following criterion (see  \cite{oguiso-truong}), which is a consequence of results in\cite{dinh-nguyen} and \cite{dinh-nguyen-truong}: If $\lambda _1(f)\not= \lambda _2(f)$ then $f$ is primitive.

The main idea behind all the results of this paper is that the existence of an automorphism $f$ of positive entropy on $X$ imposes some constraints on the cohomology groups of $X$. In fact, let $0\not=\zeta \in Nef(X)$ be such that $f^*(\zeta )=\lambda _1(f)\zeta $ (the existence of such a class is guaranteed by Perron-Frobenius theorem). The differential $df$ gives an isomorphism between the tangent bundle $TX$ and its pullback $f^*(TX)$. Hence, from the properties of Chern classes we have $f^*c_1(X)=c_1(X)$ and $f^*c_2(X)=c_2(X)$. Since $\lambda _1(f)>1$ and $X$ has dimension $3$, it follows that $$\zeta ^3=\zeta ^2.c_1(X)=\zeta .c_1(X)^2=0.$$ In fact, stronger constraints are satisfied. 
\begin{theorem}
Let $X$ be a projective manifold of dimension $3$ and $f:X\rightarrow X$ an automorphism. 

1) If $f$ has positive entropy, there is a nef class $\zeta $ which is not in $\mathbb{R}.Pic _{\mathbb{Q}}(X)$ such that $\zeta ^2=0$, $\zeta .c_1(X)^2=0$ and $\zeta .c_2(X)=0$.

2) If $\lambda _1(f)\not= \lambda _2(f)$, there is a nef class $\zeta$ which is not in $\mathbb{R}.Pic _{\mathbb{Q}}(X)$ such that $\zeta ^2=0$, $\zeta .c_1(X)=0$ and $\zeta .c_2(X)=0$. 

\label{TheoremNonExistenceClassB}\end{theorem}
Here we comment on the condition $\zeta ^2=0$. If $X$ has dimension $2$, then this condition is one homogeneous equation in $m$ variables (here, $m$ is the Picard number of $X$)  and hence is very easily satisfied.  In contrast, when $X$ has dimension $3$ or bigger, then the condition $\zeta ^2=0$ is a system of $p\geq m$ homogeneous equations in the $m$ variables (here $p$ is the dimension of $\bigwedge ^2Pic _{\mathbb{R}}(X)$), and hence is more difficult to be satisfied. This is a heuristic argument for why it is difficult to find automorphisms of positive entropy in dimension $3$ or larger.   

Based on Theorem \ref{TheoremNonExistenceClassB}, we state some conditions on nef cohomology classes. 

\begin{condition} Let $X$ be a smooth projective threefold. 

1) Condition A: We say that $X$ satisfies Condition A if whenever $\zeta \in Nef (X)$ is such that $\zeta ^2=0$, $\zeta .c_1(X)^2\geq 0$ and $\zeta .c_2(X)\leq 0$, then $\zeta \in \mathbb{R}.Pic _{\mathbb{Q}}(X)$.

2) Condition B: We say that $X$ satisfies Condition B if whenever $\zeta \in Nef (X)$ is such that $\zeta ^2=0$, $\zeta .c_1(X)= 0$ and $\zeta .c_2(X)\leq 0$, then $\zeta \in \mathbb{R}.Pic _{\mathbb{Q}}(X)$. 
\label{Conditions}\end{condition}

By Theorem \ref{TheoremNonExistenceClassB}, if $X$ satisfies Condition A then any automorphism on $X$ has zero entropy, and if $X$ satisfies Condition B then for any automorphism $f$ of $X$ we have $\lambda _1(f)=\lambda _2(f)$. While requiring more than the assumptions in part 1) of Theorem \ref{TheoremNonExistenceClassB}, Condition A is very suitable for inductive arguments. A similar comment applies for Condition B. 

Now we are ready to state the main results of this paper. The first result is for blowups of some special configurations of $\mathbb{P}^3$. 
\begin{theorem}
Let $p_1,\ldots ,p_n$ be distinct points in $X_0=\mathbb{P}^3$ such that any $4$ points of them do not belong to the same hyperplane. Let  $C_{i,j}$ be the line connecting the points $p_i$ and $p_j$. Let $\pi _1:X_1\rightarrow \mathbb{P}^3$ be the blowup at $p_1,\ldots ,p_n$. Let $D_{i,j}\subset X_1$ be the strict transforms of $C_{i,j}$, and $\pi _2:X_2\rightarrow X_1$ be the blowup at $D_{i,j}$. Then any automorphism of $X_2$ has zero entropy. 
\label{TheoremResolutionTheMapJ}\end{theorem}
{\bf Remark.} Igor Dolgachev and Yuri Prokhorov informed us that in the special cases where $4\leq n\leq 7$, then the automorphism group of $X_2$ in Theorem \ref{TheoremResolutionTheMapJ} is finite. The conclusion of Theorem \ref{TheoremResolutionTheMapJ} can be proved for the blowups of more general configurations in $\mathbb{P}^3$. However, since the statements of these generalizations are a bit complicated, we refer to Section \ref{SectionExamples} for more details.

The next two main results of the paper are for threefolds having Picard number $1$ or satisfying a special property on the second Chern class. We recall that a class $\zeta$ on $X$ is movable if there is a smooth blowup $\pi :Z\rightarrow X$ such that $\zeta$ is the pushforward of some nef class on $Z$. 
\begin{theorem}
Let $X_0$ be a threefold with Picard number $1$. Let $C_1,\ldots ,C_t\subset X_0$ be smooth curves which are pairwise disjoint. Let $p_1,\ldots ,p_s\in X_0$ be distinct points, which are allowed to belong to the curves $C_1,\ldots ,C_t$.  Let $\pi _1:X_1\rightarrow X_0$ be the blowup at $p_1,\ldots ,p_s$, and $\pi _2:X_2\rightarrow X_1$ the blowup at $C_1,\ldots ,C_t$. Then $X_2$ satisfies Properties A and B.  
\label{Theorem1}\end{theorem}
We note that in general Theorem \ref{Theorem1} does not hold for threefolds $X_0$ with Picard number $\geq 2$ (for example when $X_0=\mathbb{P}^2\times \mathbb{P}^1$). However, the theorems below may still hold for those manifolds. See Section \ref{SectionExamples} for more details. 

\begin{theorem}
Let $X_0$ be a smooth projective threefold such that $c_2(X_0).\zeta >0$ for all non-zero movable $\zeta \in NS_{\mathbb{R}}(X_0)$. Let $p_1,\ldots ,p_n\in X_0$ be distinct points. Let $\pi _1:X_1\rightarrow X_0$ be the blowup of $X_0$ at $p_1,\ldots , p_n$. Let $D_1,\ldots , D_m\subset X_1$ be disjoint smooth curves, and $\pi _2:X_2\rightarrow X_1$ the blowup at $D_1,\ldots ,D_n$.

1) $X_2$ satisfies Condition B. 

2) Assume moreover that for any $j$, then $c_1(X_1).D_j\leq 2g_j-2$, where $g_j$ is the genus of $D_j$. Then $X_2$ satisfies Condition A.  
\label{Theorem2}\end{theorem}

Theorem \ref{Theorem2} applies for $X_0=\mathbb{P}^3$ or $\mathbb{P}^2\times \mathbb{P}^1$ or $\mathbb{P}^1\times \mathbb{P}^1\times \mathbb{P}^1$. It also applies for complete intersection threefolds in $\mathbb{P}^N$. See Section \ref{SectionExamples} for more details. We note that here the images in $X_0$ of $D_1,\ldots ,D_n$ may be singular and intersect with each other, hence Theorem \ref{Theorem2} is not covered by Theorem \ref{Theorem1} even in the case $X_0=\mathbb{P}^3$.

Finally, we state several results which are purely inductive in nature, which can be applied to blowups of Fano threefolds as well. Here we recall that a threefold is Fano if $c_1(X)$ is ample.  

We start with the case of point blowups. 
\begin{theorem}
Let $Y$ be a smooth projective threefold satisfying one of the Conditions A and B. Let $\pi :X\rightarrow Y$ be the blowup at a point. Then $X$ satisfies the same Condition. 
\label{Theorem3}\end{theorem}

Next, we consider the case of curve blowups. 
\begin{theorem}
Let $Y$ be a smooth projective threefold satisfying Condition A or B. Let $\pi :X\rightarrow Y$ be the blowup at a smooth curve $C\subset Y$. Let $g$ be the genus of $C$, and define $\gamma =c_1(Y).C+2g-2$. Then $X$ also satisfies the same Condition, if one of the following cases happens.

1) $c_1(Y).C$ is an odd number and the normal vector bundle $N_{C/Y}$ is decomposable. The latter means that $N_{C/Y}$ is the direct sum of two line bundles over $C$.

2) $\gamma <0$ and $C$ is not the only effective curve in its cohomology class. 

3) There is an irreducible hypersurface $S\subset Y$ such that  $2\kappa <\mu \gamma$. Here $\kappa =S.C$ and $\mu$ is the multiplicity of $C$ in $S$.
\label{Theorem4}\end{theorem}

We note that in 1) of Theorem \ref{Theorem4}, the condition that $N_{C/Y}$ is decomposable may be easily to satisfy. For example, if $C$ is a smooth rational curve, then $N_{C/Y}$ is always decomposable by a result of Grothendieck, even if $C$ does not move in $Y$.

\begin{theorem}
Let $Y$ be a smooth projective threefold satisfying Condition B. Let $\pi :X\rightarrow Y$ be the blowup at a smooth curve $C\subset Y$. Let $g$ be the genus of $C$. If $c_1(Y).C\not= 2g-2$, then $X$ also satisfies Condition B.
\label{Theorem5}\end{theorem}
Hence, we conclude that if $X_0$ is a smooth threefold which is Fano or has Picard number $1$, then for almost every $X\rightarrow X_0$ a finite composition of points or smooth curves, every automorphism $f$ on $X$ has $\lambda _1(f)=\lambda _2(f)$. This is a strong indication that probably all automorphisms on such manifolds are not primitive, i.e. has invariant fibrations over a base of dimension $1$ or $2$. 

In Section \ref{SectionPossibleApplication} we will discuss possible application of the above results to the Ueno's threefold considered in \cite{ueno}. In Section \ref{SectionExamples}, we will give various examples illustrating the above results. 

{\bf Remark.} The general case of compact K\"ahler threefolds can be similarly treated, by replacing the Neron-Severi group by the $(1,1)$ cohomology group. After the appearance of a first version of this paper (see \cite{truong}), some generalizations to higher dimensions have been given in \cite{bayraktar-cantat} and \cite{truong1}.

{\bf Acknowledgements.} The author is grateful to Tien-Cuong Dinh for his suggestion that the answer to Question 1 is negative. The author has been benefited from helpful discussions and correspondences with Ekaterina Amerik, Turgay Bayraktar, Eric Bedford, Frederic Campana, Igor Dolgachev, Mattias Jonsson, Jan-Li Lin Viet-Anh Nguyen, Keiji Oguiso, Yuri Prokhorov, Roland Roeder and Konstantin Shramov. 

\section{Preliminaries on nef classes and blowups}
\label{SectionPreliminaries}

\subsection{K\"ahler, nef and psef classes, and effective varieties}

Let $X$ be a compact K\"ahler manifold. Let $\eta \in H^{1,1}(X)$. We say that $\eta $ is K\"ahler if it can be represented by a K\"ahler $(1,1)$ form. We say that $\eta $ is nef if it is a limit of a sequence of K\"ahler classes. We say that $\eta$ is psef if it can be represented by a positive closed $(1,1)$ current. A class $\xi \in H^{p,p}(X)$ is an effective variety if there are irreducible varieties $C_1,\ldots ,C_t$ of codimension $p$ in $X$ and non-negative real numbers $a_1,\ldots ,a_t$ so that $\xi$ is represented by $\sum _{i}a_iC_i$.

Demailly and Paun \cite{demailly-paun} gave a characterization of K\"ahler and nef classes, which in the case of projective manifolds is summarized as follows:
\begin{theorem}
Let $X$ be a projective manifold with a K\"ahler $(1,1)$ form $\omega$. A class $\eta \in H^{1,1}(X)$ is K\"ahler if and only for any irreducible subvariety $V\subset X$ then $\int _V\eta ^{dim(V)}>0$.  A class $\eta \in H^{1,1}(X)$ is nef if and only for any irreducible subvariety $V\subset X$ then $\int _V\eta ^{dim(V)-j}\wedge \omega ^j\geq 0$ for all $0\leq j\leq dim (V)$.  
\label{TheoremDemaillyPaun}\end{theorem}

Nef classes are preserved under pullback by holomorphic maps. 
\begin{lemma}
Let $\pi :X\rightarrow Y$ be a holomorphic map between compact K\"ahler manifolds. Then $\pi ^*(H^{1,1}_{nef}(X))\subset H^{1,1}_{nef}(Y)$.
\label{LemmaPullbackNefClasses}\end{lemma}
\begin{proof}
Since nef classes are in the closure of K\"ahler classes, it suffices to show that if $\eta$ is a K\"ahler class then $\pi ^*(\eta )$ is nef. Let $\varphi $ be a K\"ahler $(1,1)$ form representing $\eta$. Then $\pi ^*(\varphi )$ is a positive smooth $(1,1)$ form. Let $\omega _X$ be a K\"ahler $(1,1)$ form on $X$. Then $\pi ^*(\eta )$ is represented as a limit of the following K\"ahler classes
\begin{eqnarray*} 
\pi ^*(\varphi )+\frac{1}{n}\omega _X,
\end{eqnarray*}
and hence is nef.
\end{proof}

{\bf Remark:} Similarly, it can be shown that psef classes are preserved under pushforward by holomorphic maps. However, nef classes may not be preserved under pushforwards, even when the map is a blowup.  

\subsection{Blowup of a projective $3$-manifold at a point}

Let $\pi :X\rightarrow Y$ be the blowup of a projective $3$-manifold at a point $p$. 
Let $E= \mathbb{P}^2$ be the exceptional divisor and let $L\subset E$ be a line. Then $H^{1,1}(X)$ is generated by $\pi ^*(H^{1,1}(Y))$ and $E$, and $H^{2,2}(X)$ is generated by $\pi ^*(H^{2,2}(Y))$ and $L$. The intersection product on the cohomology of $X$ is given by 
\begin{eqnarray*}
\pi ^*(\xi ).E=0,~E.E=-L,\\
\pi ^*(\xi ).L=0,~E.L=-1. 
\end{eqnarray*}

The first and second Chern classes of $X$ can be computed by (see e.g. Section 6, Chapter 4 in the book of Griffiths-Harris \cite{griffiths-harris})
\begin{eqnarray*}
c_1(X)&=&\pi ^*(c_1(Y))-2E,\\
c_2(X)&=&\pi ^*(c_2(Y)).
\end{eqnarray*}

The following result concerns the relations between cycles on $X$ and $Y$.
\begin{lemma}
 For any effective curve $V\subset Y$, there is an effective curve $\widetilde{V}\subset X$ so that $\pi _*(\widetilde{V})=V$ and $\widetilde{V}.E\geq 0$.  
\label{LemmaPushforwardCyclesBlowupPoint}\end{lemma}  
\begin{proof}
It suffices to consider the case when $V$ is an irreducible curve. We can choose $\widetilde{V}$ to be the strict transform of $V$. Then $\pi _*(\widetilde{V})=V$, and $\widetilde{V}$ is not contained in $E$. Therefore $\widetilde{V}.E\geq 0$.  
\end{proof}

We end this subsection showing that nef classes are preserved under pushforward by point-blowups. 
\begin{lemma}
Let $\eta \in H^{1,1}_{nef}(X)$. Then $\pi _*(\eta )\in H^{1,1}_{nef}(Y)$.
\label{LemmaPushforwardNefBlowupPoint}\end{lemma}
\begin{proof}
It suffices to prove the conclusion when $\eta$ is a K\"ahler class. Let $\varphi$ be a K\"ahler $(1,1)$ form representing $\eta$. Then $\pi _*(\varphi )$ is a positive closed $(1,1)$ current, which is smooth on $X-p$. 

Let $\omega _Y$ be a K\"ahler $(1,1)$ form on $Y$. To show that $\pi _*(\eta )$ is a nef class, by Theorem \ref{TheoremDemaillyPaun} it suffices to show that for any irreducible variety $V\subset Y$ then $\pi _*(\eta )^{dim(V)-j}.V.\omega _Y^j\geq 0$ for $0\leq j\leq dim (V)$. We let $[V]$ be the current of integration on $V$. Then by the results in Section 4, Chapter 3 in the book of Demailly \cite{demailly}, the current $\pi _*(\varphi )^{dim(V)-j}\wedge [V]\wedge \omega _Y^j$ is well-defined and is a positive measure, whose mass equals to $\pi _*(\eta )^{dim(V)-j}.V.\omega _Y^j$. Thus the latter quantity is non-negative.  
\end{proof}

\subsection{Blowup of a projective $3$-manifold along a smooth curve}
Let $\pi :X\rightarrow Y$ be the blowup of a projective $3$-manifold along a smooth curve $C\subset Y$. Let $g$ be the genus of $C$. Let $F$ be the exceptional divisor and let $M$ be a fiber of the projection $F\rightarrow C$. We can identify $F$ with the projective bundle $\mathbb{P}(\mathcal{E})\rightarrow C$, where $\mathcal{E}=N_{C/Y}\rightarrow C$ is the normal vector bundle of $C$ in $Y$.   

Then $H^{1,1}(X)$ is generated by $\pi ^*(H^{1,1}(Y))$ and $F$, and $H^{2,2}(X)$ is generated by $\pi ^*(H^{2,2}(Y))$ and $M$. The intersection between $F$ and $M$ is $F.M=-1$. The first and second Chern classes of $X$ can be computed as follows:
\begin{eqnarray*}
c_1(X)&=&\pi ^*(c_1(Y))-F,\\
c_2(X)&=&\pi ^*(c_2(Y)+C)-\pi ^*c_1(Y).F.
\end{eqnarray*}

Let $[F]\rightarrow X$ be the line bundle of $F$ in $X$, and denote by $e=[F]|_F$. Then (see e.g. Section 6, Chapter 4 in the book of Griffiths - Harris \cite{griffiths-harris}) in $F$ we have the equalities 
\begin{eqnarray*}
e.M=-1,~e.e=-c_1(\mathcal{E}).
\end{eqnarray*}  
From the SES of vector bundles on $C$ 
\begin{eqnarray*}
0\rightarrow T_C\rightarrow T_Y|_C\rightarrow \mathcal{E}\rightarrow 0,
\end{eqnarray*}
it follows by the additivity of first Chern classes that 
\begin{eqnarray*}
c_1(\mathcal{E})=c_1(T_Y).C-c_1(T_C)=c_1(Y).C+2g-2.
\end{eqnarray*}
We define 
\begin{eqnarray*}
\gamma :=c_1(Y).C+2g-2.
\end{eqnarray*} 
Since $F\rightarrow C$ is a ruled surface (i.e. its fibers are projective lines $\mathbb{P}^1$), there is a canonical section $C_0$ which is the image of a holomorphic map $\sigma _0:C\rightarrow F$ (see e.g. Section 2, Chapter 5 in Hartshorne's book \cite{hartshorne}). Therefore $C_0$ is an effective curve in $F$. Such a $C_0$ has intersection $1$ with a fiber $M$.

We will return to the canonical section $C_0$ at the end of this subsection. For now, we however work in a more general assumption on $C_0$, for using later. That is, we consider an effective curve $C_0\subset F$ with the following properties  
\begin{eqnarray*}
C_0.C_0&=&\tau ,\\
C_0.M&=&\mu >0,\\
M.M&=&0.
\end{eqnarray*}   
Any divisor on $F$ is numerically equivalent to a linear combination of $C_0$ and $M$. We now show the following 

\begin{lemma}

a) 
\begin{equation}
F.C_0=\frac{1}{2}(\gamma \mu -\frac{\tau}{\mu}).
\label{Equation1}\end{equation}

b) \begin{eqnarray*}
F.F=-\frac{1}{\mu}C_0+\frac{1}{2}(\frac{\tau}{\mu ^2}+\gamma )M.
\end{eqnarray*}

c) $\pi _*(F.F)=-C$.
\label{LemmaIntersectionOnF}\end{lemma}

\begin{proof} 

a) In fact, we have
\begin{eqnarray*}
F.C_0=[F]|_{C_0}=[F]|_{F}.C_0=e.C_0,
\end{eqnarray*}
here the two expressions on the RHS are computed in $F$. On $F$, numerically we can write $e=aC_0+bM$. Then from $-1=e.M=(aC_0+bM).M=a\mu$, we get $a=-1/\mu$. Substitute  this into $e.e=-\gamma$ we obtain 
\begin{eqnarray*}
-\gamma =e.e=(\frac{1}{\mu}C_0-bM).(\frac{1}{\mu}C_0-bM)=\frac{\tau}{\mu ^2}-2b, 
\end{eqnarray*} 
which implies that 
\begin{eqnarray*}
b=\frac{1}{2}(\frac{\tau}{\mu ^2}+\gamma ). 
\end{eqnarray*}

Therefore 
\begin{eqnarray*}
e=\frac{-1}{\mu}C_0+\frac{1}{2}(\frac{\tau}{\mu ^2}+\gamma )M.
\end{eqnarray*}
Thus
\begin{eqnarray*}
F.C_0&=&e.C_0=[\frac{-1}{\mu}C_0+\frac{1}{2}(\frac{\tau}{\mu ^2}+\gamma )M]C_0\\
&=&\frac{-\tau}{\mu}+\frac{1}{2}(\frac{\tau}{\mu}+\gamma \mu )\\
&=&\frac{1}{2}(-\frac{\tau}{\mu}+\gamma \mu ).
\end{eqnarray*}

b) From the formula for $e$ in the proof of a) it is not difficult to arrive at the proof of b).

c) Since $C_0.M=\mu$, it follows that $\pi _*(C_0)=\mu C$. Then from b) we obtain c). 
\end{proof}

We end this subsection commenting on conditions 2) and 3) of Theorem \ref{TheoremAutomorphismBlowupP3}. By Proposition 2.8 in Chapter 5 of \cite{hartshorne}, there is a line bundle $\mathcal{M}\rightarrow C$ so that the vector bundle $\mathcal{E}'=\mathcal{E}\otimes \mathcal{M}$ is normalized in the following sense: $H^0(\mathcal{E}')\not= 0$ but for all line bundle $\mathcal{L}\rightarrow C$ with $c_1(\mathcal{L})<0$ then $H^0(\mathcal{E}'\otimes \mathcal{L})=0$. A canonical section $C_0\subset F$ can be associated to such a normalized $\mathcal{E}'$. The intersection between $C_0$ and $M$ is $1$. Moreover, the number 
\begin{eqnarray*}
\tau _0=C_0.C_0=c_1(\mathcal{E}')=c_1(\mathcal{E})+2c_1(\mathcal{M}),   
\end{eqnarray*}
is an invariant of $F$.

We end this section with some further properties of a ruled surface. 

\begin{lemma}
Assume that the invariant $\tau _0$ of $F$ is non-negative. Then

a) For any effective curve $V\subset F$ we have $V.V\geq 0$.

b) If moreover $\gamma <0$ then for any non-zero effective curve $V\subset F$ we have $F.V<0$.
\label{LemmaCaseTauNonNegative}\end{lemma}  
\begin{proof}

a) It suffices to prove for the case $V$ is an irreducible curve. Numerically, we write $V=aC_0+bM$. If $V=C_0$ then $V.V=\tau _0\geq 0$. If $V=M$ then $V.V=0$. Hence we may assume that $V\not= C_0,M$. 

We consider two cases:

Case 1: $\tau _0=0$. By Proposition 2.20 in Chapter 5 in \cite{hartshorne}, we have $a>0$ and $b\geq 0$. Therefore
\begin{eqnarray*}
 V.V=a^2\tau _0+2ab\geq 0.
\end{eqnarray*}

Case 2: $\tau _0>0$. By Proposition 2.21 in Chapter 5 in \cite{hartshorne}, there are two subcases:

Subcase 2.1: $a=1,b\geq 0$. Then 
\begin{eqnarray*}
V.V=\tau _0+2b\geq 0.
\end{eqnarray*}

Subcase 2.2: $a\geq 2,b\geq -a\tau _0/2$. Then
\begin{eqnarray*}
V.V=a^2\tau _0+2ab\geq a^2\tau _0+2a(-a\tau _0/2)=0.
\end{eqnarray*}

b) It suffices to prove for the case $V$ is an irreducible curve. If $V=M$ then $F.M=-1<0$. If $V=C_0$ then by Lemma \ref{LemmaIntersectionOnF} with $\tau =\tau _0\geq 0$ and $\mu =1$
\begin{eqnarray*}
F.C_0=\frac{1}{2}(\gamma -\tau _0)\leq \frac{1}{2}\gamma <0
\end{eqnarray*}
because $\gamma <0$. Therefore we may assume that $V\not= C_0,M$, and then proceed as in the proof of a). 
\end{proof}

\section{Proofs of the main results}
\label{SectionMainResults}

We make use of the following result (see e.g. \cite{zhang3} and \cite{bedford}).
\begin{lemma}
Let $X$ be a smooth projective threefold. Let $f:X\rightarrow X$ be an automorphism. If $\lambda _1(f)>1$, then $\lambda _1(f)$ is irrational.  
\label{LemmaLambdaIrrational}\end{lemma}
 For the convenience of the readers, we reproduce the proof of this Lemma here. 
 \begin{proof}
 Let $A$ be the matrix of $f^*:NS_{\mathbb{R}}(X)\rightarrow NS_{\mathbb{R}}(X)$, then $A$ is an integer matrix, and $\lambda _1(f)$ is a real eigenvalue of $A$. Moreover, $A$ is invertible and its inverse $A^{-1}$ is the matrix of the map $(f^{-1})^*:NS_{\mathbb{R}}(X)\rightarrow NS_{\mathbb{R}}(X)$ hence is also an integer matrix. Therefore $det(A)=\pm 1$. Thus the characteristic polynomial $P(x)$ of $A$ is a monic polynomial of integer coefficients and $P(0)=\pm 1$. Assume that $\lambda _1(f)$ is a rational number. Since $\lambda _1(f)$ is an algebraic integer, it follows that $\lambda _1(f)$ must be an integer. Then we can write $P(x)=(x-\lambda _1(f))Q(x)$, here $Q(x)$ is a polynomial of integer coefficients. If $\lambda _1(f)>1$ we get a contradiction $\pm 1=P(0)=-\lambda _1(f)Q(0)$
 \end{proof}

Now we give the proofs of the main results. 

\begin{proof}[Proof of Theorem \ref{TheoremNonExistenceClassB}]
1)  Since $f^*$ preserves the cone $Nef (X)$, by a Perron-Frobenius type theorem, there is a non-zero nef class $\eta $ so that $f^*(\eta )=\lambda _1(f)\eta $. Similarly, there is a non-zero nef class $\eta _{-}$ so that $(f^{-1})^*(\eta _{-})=\lambda _1(f^{-})\eta _{-}$. 

Assume that $\lambda _1(f)>1$. By the log-concavity of dynamical degrees, we also have $\lambda _1(f^{-1})>1$.  By Lemma \ref{LemmaLambdaIrrational}, both $\lambda _1(f)$ and $\lambda _1(f^{-1})$ are irrational. Hence, both $\zeta$ and $\zeta _{-}$ are not in $\mathbb{R}.NS_{\mathbb{Q}}(X)$. It is easy to see that
\begin{eqnarray*}
\zeta .c_1(X)^2&=&\zeta .c_2(X)=0,\\
\zeta _{-}.c_1(X)^2&=&\zeta _{-}.c_2(X)=0.
\end{eqnarray*}

To prove 1) it suffices to show that either $\zeta ^2=0$ or $\zeta _{-}^2=0$. Assume otherwise. From $\zeta ^2\not=0$ and $f^*(\zeta ^2)=\lambda _1(f)^2\zeta ^2$, we have $$\lambda _1(f)^2\leq \lambda _2(f)=\lambda _1(f^{-1}).$$
Similarly, from $\zeta _{-}^2\not= 0$, we have
$$\lambda _1(f^{-1})\geq \lambda _1(f)^2.$$ 
Combining these two inequalities, we conclude that $\lambda _1(f)\geq \lambda _1(f)^4$, which contradicts to $\lambda _1(f)>1$. This completes the proof of 1). 

2) The proof of 2) is similar. 
\end{proof}

\begin{proof}[Proof of Theorem \ref{TheoremResolutionTheMapJ}]
1) For the proof, it suffices to show that for any non-zero nef $\zeta$ on $X=X_2$ then either $\zeta .c_1(X)^2\not= 0$ or $\zeta .c_2(X)\not= 0$.

We let $E_1,\ldots ,E_n$ be the exceptional divisors of the blowup $\pi _1:X_1\rightarrow X_0=\mathbb{P}^3$. Let  $F_{i,j}$ be the exceptional divisors of the blowup $\pi _2:X=X_2\rightarrow X_1$. Then we can write 
\begin{eqnarray*}
\zeta &=&\pi _2^*(\xi )-\sum _{i<j}\alpha _{i,j}F_{i,j},\\
\xi &=&\pi _1^*(u)-\sum _{l}\beta _lE_l.
\end{eqnarray*}
Here $u$ is nef on $\mathbb{P}^3$ and $\alpha _{i,j},\beta _l\geq 0$. 

For the proof of 1), it then suffices to show that $\deg (u)=0$.  From 
\begin{eqnarray*}
c_2(X)=\pi _2^*c_2(X_1)+\sum _{i<j}\pi _2^*D_{i,j}-\sum _{i<j}\pi _2^*c_1(X_1).F_{i,j},
\end{eqnarray*}
and the fact that $c_1(X_1).D_{i,j}=0$, the condition $\zeta .c_2(X)=0$ becomes
$\xi .c_2(X_1)+\sum _{i<j}\xi .D_{i,j}=0$. Since $c_2(X_1)=\pi _1^*(c_2(\mathbb{P}^3))$, it follows that $\xi .c_2(X_1)=16 \deg (u)$. We also have that $\xi .D_{i,j}=\deg (u)-\beta _i-\beta _j$ for every $i<j$. Therefore, we obtain
\begin{eqnarray*}
6\deg (u)&=&-\sum _{i<j}\xi .D_{i,j},\\
(6+\frac{n(n-1)}{2})\deg (u)&=&(n-1)\sum _l\beta _l.
\end{eqnarray*}

From the condition $\zeta .c_1(X)^2=0$, we obtain
\begin{eqnarray*}
0&=&\zeta .c_1(X)^2=(\pi _2^*(\xi )-\sum _{i<j}\alpha _{i,j}F_{i,j}).(\pi _2^*c_1(X_1)^2-2\sum _{i<j}\pi _2^*c_1(X_1).F_{i,j}+\sum _{i<j}F_{i,j}^2)\\
&=&\xi . c_1(X_1)^2-\sum _{i<j}\xi .D_{i,j}-2\sum _{i<j}\alpha _{i,j}c_1(X_1).D_{i,j}+\sum _{i<j}\alpha _{i,j}(c_1(X_1).D_{i,j}+2g_{i,j}-2)\\
&=&22 \deg (u)-4\sum _l\beta _l+\sum _{i<j}\alpha _{i,j}(2g_{i,j}-2-c_1(X_1).D_{i,j})\\
&=&22\deg (u)-4\sum _{l}\beta _l-2\sum _{i<j}\alpha _{i,j}.
\end{eqnarray*}
In the above, $g_{i,j}=0$ is the genus of $C_{i,j}$, and $c_1(X_1).D_{i,j}=0$ for all $i<j$. In particular, we obtain
\begin{equation}
\frac{11}{2}\deg (u)\geq \sum _{l}\beta _l=(\frac{6}{n-1}+\frac{n}{2})\deg (u).
\label{Equation3}\end{equation}
From the above inequality, we will finish showing that $\deg (u)=0$. We consider several cases:

Case 1: $n\geq 10$. From Equation (\ref{Equation3}), it follows immediately that $\deg (u)=0$ as wanted. 

Case 2: $6\leq n\leq 9$. In this case, for each $6$ points $p_{i_1},\ldots ,p_{i_6}$ among $n$ points $p_1,\ldots ,p_n$, there is a unique rational normal curve $C\subset \mathbb{P}^3$ of degree $3$ passing through the $6$ chosen points. Let $D\subset X_1$ be the  strict transform of $C$. Then $D$ is different from the curves $D_{i,j}$. Therefore $\pi _2^*D$ is an effective curve, and hence $$3\deg (u)-\sum _{l=1}^6\beta _{i_l}\geq \xi .D=\zeta .\pi _2^*(D)\geq 0.$$  Summing over all such choices of $p_{i_1},\ldots ,p_{i_n}$ we find that
\begin{eqnarray*}
\frac{n}{2}\deg (u)\geq \sum _{l}\beta _l.
\end{eqnarray*}  
Combining this with $$\sum _l\beta _l=(\frac{6}{n-1}+\frac{n}{2})\deg (u),$$
we obtain $\deg (u)=0$.

Case 3: $n=4,5$. In this case, we use rational normal curves to obtain
\begin{eqnarray*}
\frac{n}{3}\deg (u)\geq \sum _l\beta _l.
\end{eqnarray*}
Combining this with 
\begin{eqnarray*}
\sum _{l}\beta _l=(\frac{6}{n-1}+\frac{n}{2})\deg (u),
\end{eqnarray*}
we obtain $\deg (u)=0$.

Case 4: $n=1,2,3$. In this case we have $n\deg (u)\geq \sum _l\beta _l$. Combining this with 
\begin{eqnarray*}
\sum _{l}\beta _l=(\frac{6}{n-1}+\frac{n}{2})\deg (u),
\end{eqnarray*}
we obtain $\deg (u)=0$.
  \end{proof}

\begin{proof} (Proof of Theorem \ref{Theorem1}) 

Let $\pi _1':X_1'\rightarrow X_0$ be the blowup at the $C_1,\ldots ,C_t$. Let $F_1,\ldots ,F_t$ be the exceptional divisors. Let $M_j=(\pi _1')^{-1}(p_j)$ be the preimages  of the points $p_j$ ($j=1,\ldots ,n$). These are smooth rational curves, and are among the fibers of the maps $F_1\rightarrow C_1$,  $\ldots $ , $F_t\rightarrow C_t$. Let $\pi _2':X_2'\rightarrow X_1' $ be the blowup at the curves $M_j$. Then $X_2$ is isomorphic to $X_2'$. 

Fixed a number $j$. Let $i$ be such that $M_j\subset F_i$. Using that 
\begin{eqnarray*}
c_1(N_{M_j/X_1'})=c_1(N_{M_j/F_i})+c_1(N_{F_i/X_1}|_{M_j})=0+(-1)=-1,
\end{eqnarray*} 
we find that 
\begin{eqnarray*}
c_1(X_1').M_j=1
\end{eqnarray*}
is an odd number. Therefore, using either part 1) or part 3) of Theorem \ref{Theorem4}, for the proof of Theorem \ref{Theorem2} it suffices to show that $X_1'$ satisfies both Conditions A and B. To this end, we only need to show that  if $\zeta \in Nef (X_1') $ is such that $\zeta ^2=0$ then $\zeta \in \mathbb{R}.NS_{\mathbb{Q}}(X_1')$. 

Let $H\in NS_{\mathbb{Q}}(X_0)$ be an ample divisor. Since $X_0$ has Picard number $1$, we can write
\begin{eqnarray*}
\zeta =a(\pi _1')^*(H)-\sum _j\alpha _jF_j,
\end{eqnarray*}   
where $a,\alpha _1,\ldots ,\alpha _t\geq 0$. If $a=0$, then from the fact that $\zeta$ is nef, we have $\alpha _1=\ldots =\alpha _t=0$. Therefore, we may assume that $\alpha >0$, and after dividing by $\alpha $ we may assume that $\alpha =1$. Then, for the proof of the theorem, it suffices to show that all the numbers $\alpha _1,\ldots ,\alpha _t$ are in $\mathbb{Q}$.

Since the curves $C_j$ are pairwise disjoint, for any $i=1,\ldots ,t$ we have 
\begin{eqnarray*}
0&=&\zeta ^2.F_i=((\pi _1')^*H-\sum _j\alpha _jF_j)^2.F_i\\
&=&(\pi _1')^*H^2.F_i-2\alpha _i(\pi _1')^*H.F_i^2+\alpha _i^2F_i^3\\
&=&2\alpha _iH.C_i-\alpha _i^2(c_1(X_0).C_i+2g_i-2),
\end{eqnarray*}
here $g_i$ is the genus of $C_i$. We note that $H.C_i$ is a positive rational number. Hence, either $\alpha _i=0$, or
\begin{eqnarray*}  
\alpha _i=2H.C_i/(c_1(X_0).C_i+2g_i-2).
\end{eqnarray*}
In both cases, $\alpha _i$ are rational numbers as wanted.
\end{proof}

\begin{proof}[Proof of Theorem \ref{Theorem2}]
1) Let $\zeta $ be a nef class on $X_2$ such that $\zeta ^2=0$, $\zeta .c_1(X)=0$ and $\zeta .c_1(X_2)^2\leq 0$. We need to show that $\zeta \in \mathbb{R}.H^2_{alg}(X_2,\mathbb{Q})$. More strongly, we will show that $\zeta$ must be $0$.

Let us denote by $F_j$ the exceptional divisor over $D_j$ of the blowup $\pi _2:X_2\rightarrow X_1$.  We denote by $\pi _1:X_1\rightarrow X_0$ the blowup of $C_0$ at the points $p_i$.

We can write $\zeta =\pi _2^*(\xi )-\sum _{j}\alpha _jF_j$, where $\alpha _j\geq 0$ and $\xi$ is a movable class on $X_1$. Since $D_j$ are disjoint, by intersecting the equations $\zeta ^2=\zeta .c_1(X_2)=0$ with $F_j$, we find as in \cite{truong} that either $\alpha _j=0$ or 
\begin{eqnarray*}
\xi .D_j=\alpha _jc_1(X_1).D_j=\alpha _j(2g_j-2).
\end{eqnarray*}
If $\alpha _j=0$ then $$\xi .D_j=\zeta .D_j'\geq 0=\alpha _jc_1(X_1).D_j,$$ where $D_j'\subset F_j$ is a section whose pushforward is $D_j$. If $\alpha _j\not= 0$ then $\xi .D_j=c_1(X_1).D_j$. Therefore,
\begin{eqnarray*}
0\geq \zeta .c_2(X_2)&=&(\pi _2^*(\xi )-\sum _j\alpha _jF_j).(\pi _2^*c_2(X_1)+\sum _j(\pi _2^*D_j-\pi _2^*c_1(X_j).F_j))\\
&=&\xi .c_2(X_1)+\sum _j(\xi .D_j-\alpha _jc_1(X_1).D_j).
\end{eqnarray*}
 Since each term $\xi .D_j-\alpha _jc_1(X_1).D_j$ is non-negative, we find that $\xi .c_2(X_1)\leq 0$. Because $c_2(X_1)=\pi _1^*c_2(X_0)$, we then get that $(\pi _1)_*(\xi ).c_2(X_0)\leq 0$. Because $(\pi _1)_*(\xi )$  is movable in $X_0$, from the assumption on $c_2(X_0)$ we obtain $(\pi _1)_*(\xi )=0$. From this, it easy follows that $\xi$ and then $\zeta $ are $0$.    
 
2) The proof is similar to that of 1). The difference is now that here for each $j$, either $\alpha _j=0$ or $$\xi .D_j-\alpha _jc_1(X_1).D_j=\frac{\alpha _j}{2}[(2g_j-2)-c_1(X_1).D_j].$$
In the first case
\begin{eqnarray*}
\xi .D_j-\alpha _jc_1(X_1).D_j=\xi .D_j=\zeta .D_j'\geq 0,
\end{eqnarray*}
where $D_j'\subset F_j$ is a section. In the second case, by the assumption $(2g_j-2)-c_1(X_1).D_j\geq 0$, we also have $\xi .D_j-\alpha _jc_1(X_1).D_j\geq 0$. 

Hence,
\begin{eqnarray*}
0\geq -\sum _j(\xi .D_j-\alpha _jc_1(X_1).D_j)\geq \xi .c_2(X_1).
\end{eqnarray*}
Then we can proceed as before. 
\end{proof}

\begin{proof}[Proof of Theorem \ref{Theorem3}]
Let $F$ be the exceptional divisor of the blowup $\pi$. Let $\zeta$ be a nef class on $X$. Then we can write $\zeta =\pi ^*(\xi )-\alpha F$ for some $\alpha \geq 0$ and for some movable class $\xi =\pi _*(\zeta )$ on $Y$. 

Assume that $\zeta ^2=0$. Then, 
\begin{eqnarray*}
0=\zeta ^2=(\pi ^*(\xi )-\alpha F)^2=\pi ^*(\xi ^2)+\alpha ^2F^2.
\end{eqnarray*}
Here we used that $\pi ^*(\xi ).F=0$. Because the classes of $\pi ^*(\xi ^2)$ and $F^2$ are linearly independent in the $(2,2)$ cohomology group of $X$, from the above we have that $\alpha =0$. Then it follows that $\xi$ is nef on $Y$, and $\xi ^2=0$. Moreover, since $c_1(X)=\pi ^*c_1(Y)-2F$ and $c_2(X)=\pi ^*c_2(Y)$ (see Chapter 4 in \cite{griffiths-harris}), we have
\begin{eqnarray*}
\xi .c_1(Y)&=&\pi _*(\pi ^*(\xi ).\pi ^*(c_1(Y)))=\pi _*(\pi ^*(\xi ).(\pi ^*c_1(Y)-2F))=\pi _*(\zeta .c_1(X)),\\
\xi .c_1(Y)^2&=&\zeta .c_1(X)^2,\\
\xi .c_2(Y)&=&\zeta .c_2(X).
\end{eqnarray*}
Then, it follows easily that if $Y$ satisfies one of the Conditions A and B, then $X$ also satisfies the same Condition.
\end{proof}

\begin{proof}[Proof of Theorem \ref{Theorem4}] We will show that if $Y$ satisfies Condition A then $X$ also satisfies Condition A. The proof for Condition B is similar.  

Let $\zeta $ be a nef class on $X$. We need to show that if 
\begin{eqnarray*}
\zeta ^2&=&0,\\
\zeta .c_1(X)^2&\geq &0,\\
\zeta .c_2(X)&\leq &0,
\end{eqnarray*}
then $\zeta \in \mathbb{R}.NS_{\mathbb{Q}}(X)$.

Let $F$ be the exceptional divisor of the blowup $\pi :X\rightarrow Y$.  We can write $\zeta =\pi ^*(\xi )-\alpha F$ for some $\alpha \geq 0$. We also have (see Section \ref{SectionPreliminaries})
\begin{eqnarray*}
c_1(X)&=&\pi ^*c_1(Y)-F,\\
c_2(X)&=&\pi ^*c_2(Y)+\pi ^*C-\pi ^*c_1(Y).F,\\
\pi _*(F.F)&=&-C.
\end{eqnarray*}

We first consider the case $\alpha =0$. Then, $\xi$ is nef on $Y$ and moreover $\xi ^2=0$. We have in this case 
\begin{eqnarray*}
\pi _*(\zeta .c_1(X))&=&\pi _*(\pi ^*(\xi ). (\pi ^*c_1(Y)-F))=\xi .c_1(Y),\\
\zeta .c_1(X)^2&=&\pi ^*(\xi ).(\pi ^*c_1(Y)-F)^2=\pi ^*(\xi ).(\pi ^*c_1(Y)^2-2\pi ^*c_1(Y).F+F^2)\\
&=&\xi .c_1(Y)^2-\xi .C,\\
\zeta .c_2(X)&=&\pi ^*(\xi ).(\pi ^*c_2(Y)+\pi ^*C-\pi ^*c_1(Y).F)\\
&=&\xi .c_2(Y)+\xi .C.
\end{eqnarray*}
Since $\xi$ is nef and $C$ is an effective curve, we have $\xi .C\geq 0$. Therefore, from the assumptions $\zeta .c_1(X)^2\geq 0$ and $\zeta .c_2(Y)\leq 0$ we obtain
\begin{eqnarray*}
\xi ^2&=&0,\\
\xi .c_1(Y)^2&=&\zeta .c_1(X)^2+\xi .C\geq 0,\\
\xi .c_2(Y)&=&\zeta .c_2(X)-\xi .C\leq 0.
\end{eqnarray*}
Since $Y$ satisfies Condition A by assumption, it follows that $\xi \in \mathbb{R}.NS_{\mathbb{Q}}(Y)$. Then $\zeta =\pi ^*(\xi )\in \mathbb{R}.NS_{\mathbb{Q}}(X)$. Hence, $X$ also satisfies Condition A.

Now we show that under the assumptions of Theorem \ref{Theorem3}, then actually $\alpha$ must be $0$. Assume otherwise, i.e. that $\alpha >0$, we will obtain  a contradiction. We recall that $\gamma =c_1(Y).C+2g-2$. From the assumption that $\zeta ^2=0$ we have
\begin{eqnarray*}
0&=&\zeta ^2.F=(\pi ^*(\xi )-\alpha F)^2.F\\
&=&\pi ^*(\xi ^2).F-2\alpha \pi ^*(\xi ).F^2+F^3\\
&=&\alpha \xi .C-2\alpha ^2.\gamma .
\end{eqnarray*}
In the fourth equality we used the results in Section \ref{SectionPreliminaries}. The assumption that $\alpha >0$ implies that
\begin{eqnarray*}
\xi .C=\alpha .\gamma /2.
\end{eqnarray*}

We now proceed corresponding to parts 1), 2) and 3) of the theorem. 

1) In this case $c_1(Y).C$ is an odd number and $N_{C/Y}$ is decomposable. We have a SES of vector bundles over $C$:
\begin{eqnarray*}
0\rightarrow T_C\rightarrow T_Y|_C\rightarrow N_{C/Y}\rightarrow 0.
\end{eqnarray*}

From this, it follows that 
\begin{eqnarray*}
c_1(N_{C/Y})=c_1(Y).C+2g-2=\gamma .
\end{eqnarray*}

Recall that $F$ is the exceptional divisor of the blowup $\pi$. Then $F=\mathbb{P}(N_{C/X})\rightarrow C$ is a ruled surface over $C$. Hence, (see Proposition 2.8 in Chapter 5 in \cite{hartshorne}) there is a line bundle $\mathcal{M}$ over $C$ such that $\mathcal{E} =N_{C/Y}\otimes \mathcal{M}$ is normalized, in the sense that $H^0(\mathcal{E} )\not= 0$, but for every line bundle $\mathcal{L}$ with $c_1(\mathcal{L})<0$ then $H^0(\mathcal{E}\otimes \mathcal{L})=0$. 

Let $f$ be a fiber of the fibration $F\rightarrow C$. Then, (see Proposition 2.9 in Chapter 5 in \cite{hartshorne}), there is  a so-called zero section $C_0\subset F$ with the following properties: 
\begin{eqnarray*}
\tau :=C_0.C_0&=&c_1(\mathcal{E}),\\
C_0.f&=&1.
\end{eqnarray*}

Because $N_{C/Y}$ is decomposable, $\mathcal{E}$ is also decomposable. By part a) of Theorem 2.12 in Section 5 in \cite{hartshorne}, $c_1(\mathcal{E})\leq 0$. Moreover, from
\begin{eqnarray*}
c_1(\mathcal{E})=c_1(\mathcal{N_{C/Y}})+2c_1(\mathcal{M})=c_1(Y).C+2g-2+2c_1(\mathcal{M}),
\end{eqnarray*}
and the assumption that $c_1(Y).C$ is an odd number, we get that $c_1(\mathcal{E})<0$. Hence $\tau <0$. 

From the results in Section \ref{SectionPreliminaries}, we have 
\begin{eqnarray*}
C_0=-F.F+\frac{1}{2}(\tau +\gamma )f.
\end{eqnarray*}

Now we obtain the desired contradiction. Since $\zeta $ is nef and $C_0$ is an effective curve, we have $\zeta .C_0\geq 0$. Hence,
\begin{eqnarray*}
 0&\leq& (\pi ^*(\xi )-\alpha F).(-F.F+\frac{1}{2}(\tau +\gamma )f\\
 &=&\xi .\pi _*(-F.F)+\alpha F.F.F-\frac{1}{2}\alpha (\tau +\gamma )F.f\\
 &=&\xi .C-\alpha \gamma+\frac{1}{2}\alpha (\tau +\gamma )=\frac{\alpha \tau }{2}<0.
\end{eqnarray*}
In the above we used that $\pi _*(-F.F)=\pi _*(C_0)=C$ (see for example Lemma 4 in \cite{truong}), $F.f=-1$, $F.F.F=-\gamma$, $\xi .C=\alpha \gamma /2$ , $\alpha >0$ and $\tau =C_0.C_0< 0$. 

2) In this case, $\gamma <0$ and $C$ is not the only effective curve in its cohomology class. Let $D$ be another curve in the cohomology class of $C$. Since $C$ is irreducible, we can assume that $C$ is not contained in the support of $D$. Then $\pi ^*(D)$ is an effective curve in $X$. Since $\zeta $ is nef, we obtain a contradiction
\begin{eqnarray*}
0\leq \pi ^*(D).\zeta =D. \pi _*(\zeta )=D.\xi =C.\xi =\alpha \gamma /2 <0.
\end{eqnarray*}

3) In this case, there is an irreducible hypersurface $S\subset Y$ such that  $2\kappa <\mu \gamma$. Here $\kappa =S.C$ and $\mu$ is the multiplicity of $C$ in $S$. We now construct an effective curve $C_0\subset F$ and use it to derive a contradiction.

The strict transform $\widetilde{S}$ of $S$ is given by $\widetilde{S}=\pi ^*(S)-\mu F$, and is an irreducible hypersurface of $X$. Since $\widetilde{S}$ and $F$ are different irreducible hypersurfaces, their intersection $C_0=\widetilde{S}.F=(\pi ^*(S)-\mu F).F$ is an effective curve of $F$. We now compute the numbers $C_0.C_0$ and $C_0.M$. We have
\begin{eqnarray*}  
C_0.C_0&=&\widetilde{S}|_{F}.\widetilde{S}|_{F}=\widetilde{S}.\widetilde{S}.F\\
&=&(\pi ^*(S)-\mu F).(\pi ^*(S)-\mu F).F=-2\mu\pi ^*(S).F.F+\mu ^2F.F.F\\
&=&2\mu S.C-\mu ^2\gamma =2\mu \kappa -\mu ^2\gamma .
\end{eqnarray*}
Denote by $\tau =C_0.C_0$ and $\mu _0=C_0.M$. Note that $\mu _0\not=0$, otherwise we have $C_0$ is a multiplicity of $M$, and hence $\pi _*(C_0)=0$. But from the definition of $C_0$ we can see that $\pi _*(C_0)=\mu C\not= 0$. Then by the computations in Section \ref{SectionPreliminaries}, we have
\begin{eqnarray*}
F.F=-\frac{1}{\mu _0}C_0+\frac{1}{2}(\frac{\tau }{\mu _0^2}+\gamma )M.
\end{eqnarray*}
Pushforward this by the map $\pi$, using that $\pi _*(F.F)=-C$ and $\pi _*(C_0)=\mu C$ we have that $\mu _0=\mu$.  

From the above computation $\tau =2\mu \kappa -\mu ^2\gamma $, we obtain
\begin{eqnarray*}
F.C_0=\frac{1}{2}(\gamma \mu -\frac{\tau}{\mu})=\gamma \mu -\kappa .
\end{eqnarray*}

Because $\zeta$ is nef, it follows that 
\begin{eqnarray*}
0&\leq&\zeta .C_0=(\pi ^*(\xi )-\alpha F).C_0=\mu \xi .C-\frac{\alpha }{2}(\gamma \mu -\frac{\tau}{\mu}),\\
&=&\frac{\alpha}{2}\gamma \mu -\frac{\alpha}{2}(\gamma \mu -\frac{\tau}{\mu})=\frac{\alpha}{2}\frac{\tau}{ \mu} =\alpha (\kappa -\frac{1}{2}\gamma \mu  ).
\end{eqnarray*}
This contradicts the assumptions that $2\kappa <\gamma \mu$ and $\alpha >0$. 
\end{proof}

\begin{proof} [Proof of Theorem \ref{Theorem5}]

Let $\zeta $ be a nef class on $X$. We need to show that if 
\begin{eqnarray*}
\zeta ^2&=&0,\\
\zeta .c_1(X)&= &0,\\
\zeta .c_2(X)&\leq &0,
\end{eqnarray*}
then $\zeta \in \mathbb{R}.NS_{\mathbb{Q}}(X)$.

Let $F$ be the exceptional divisor of the blowup $\pi :X\rightarrow Y$.  We can write $\zeta =\pi ^*(\xi )-\alpha F$ for some $\alpha \geq 0$. As in the proof of Theorem \ref{Theorem4}, it suffices to show that $\alpha =0$. We assume otherwise that $\alpha >0$. Let $f\subset F$ be a fiber of the projection $F\rightarrow C$. We have 
\begin{eqnarray*}
0&=&\zeta .\zeta =(\pi ^*(\xi )-\alpha F).(\pi ^*(\xi )-\alpha F)\\
&=&\pi ^*(\xi .\xi )-2\alpha \pi ^*(\xi ).F+\alpha ^2F.F,\\
0&=&\zeta .c_1(X)=(p^*(\xi )-\alpha F).(\pi ^*c_1(Y)-F)\\
&=&\pi ^*(\xi .c_1(Y))-\pi ^*(\xi ).F-\pi ^*c_1(Y).F+\alpha F^2.
\end{eqnarray*}
Intersecting both of these equations with $F$, using $F.F.F=-\gamma$ and $\pi _*(F.F)=-C$, we obtain
\begin{eqnarray*}
&&2\alpha \xi .C-\alpha ^2\gamma =0,\\
&&\alpha c_1(Y).C+\xi .C-\alpha \gamma =0.
\end{eqnarray*}
Then we must have $\alpha =0$. Otherwise, dividing $2\alpha$ from the first equation we have that $\xi .C=\alpha \gamma /2$. Substituting this into the second equation and dividing by $\alpha$ we get $2c_1(Y).C=\gamma$. Hence $c_1(Y).C=2g-2$, which is a contradiction.
\end{proof}

\section{Examples}
\label{SectionExamples}
 
\subsection{The case $X_0=\mathbb{P}^2\times \mathbb{P}^1$}
The Picard number of $X_0$ is $2$. By K\"unneth's formula, $H^{1,1}(X_0)$ is generated by the classes of $\mathbb{P}^2\times \{pt\}$ and $\mathbb{P}^1\times \mathbb{P}^1$ (here $\{pt\}$ means a point). The intersection on $H^{1,1}(X_0)$ is 
\begin{eqnarray*}
\mathbb{P}^2\times \{pt\}.\mathbb{P}^2\times \{pt\}&=&0,\\
\mathbb{P}^2\times \{pt\}.\mathbb{P}^1\times \mathbb{P}^1&=&\mathbb{P}^1\times \{pt\},\\
\mathbb{P}^1\times \mathbb{P}^1.\mathbb{P}^1\times \mathbb{P}^1&=&\{pt\}\times \mathbb{P}^1.
\end{eqnarray*}
By K\"unneth's formula again, $H^{2,2}(X_0)$ is generated by $\mathbb{P}^1\times \{pt\}$ and $\{pt\}\times \mathbb{P}^1$. The pairing between $H^{1,1}(X_0)$ and $H^{2,2}(X_0)$ is given by
\begin{eqnarray*}
\mathbb{P}^2\times \{pt\}.\mathbb{P}^1\times \{pt\}&=&0,\\
\mathbb{P}^2\times \{pt\}.\{pt\}\times \mathbb{P}^1&=&1,\\
\mathbb{P}^1\times \mathbb{P}^1.\mathbb{P}^1\times \{pt\}&=&1,\\
\mathbb{P}^1\times \mathbb{P}^1.\{pt\}\times \mathbb{P}^1&=&0.
\end{eqnarray*}

By Whitney's formula, we have
\begin{eqnarray*}
c_1(X_0)&=&2\mathbb{P}^2\times \{pt\}+3\mathbb{P}^1\times \mathbb{P}^1,\\
c_2(X_0)&=&6\mathbb{P}^1\times \{pt\}+3\{pt\}\times \mathbb{P}^1.
\end{eqnarray*} 

Therefore, we can check that $X_0$ satisfies all the conditions of Theorems  \ref{Theorem2}, \ref{Theorem3}, \ref{Theorem4} and \ref{Theorem5}. In particular, if $D_1,\ldots ,D_n\subset X_0$ are pairwise disjoint smooth curves, and $\pi _1:X_1\rightarrow X_0$ is the blowup at $D_1,\ldots ,D_n$, then for any automorphism $f$ of $X_1$  we have $\lambda _1(f)=\lambda _2(f)$. However, $X_0$ does not satisfy the conditions of Theorem \ref{Theorem1}, its Picard number is $2>1$. For an appropriate choice of curves $D_1,\ldots ,D_n$, the threefold $X_1$ has automorphisms of positive entropy. In fact, there is a rational surface $S$ obtained from $\mathbb{P}^2$ by blowing up distinct points $p_1,\ldots ,p_n\in \mathbb{P}^2$ such that $S$ has an automorphism of positive entropy. If we choose $D_j=p_j\times \mathbb{P}^1$, then $D_j$ are smooth rational curves which are disjoint, and $X_1$ has an automorphism of positive entropy.

\subsection{The case $X_0=\mathbb{P}^1\times \mathbb{P}^1\times \mathbb{P}^1$} This case is very similar to the case $X_0=\mathbb{P}^2\times \mathbb{P}^1$ above. The readers can easily redo all the (analogs of) computations and constructions in the previous section. 
 
\subsection{The case $X_0=$ a complete intersection in $\mathbb{P}^N$}
Let $X_0$ be a smooth projective threefold which is a complete intersection in $\mathbb{P}^N$. This means that $X_0$ is the intersection of smooth hypersurfaces $D_1,\ldots ,D_{N-3}$ of $\mathbb{P}^N$. By Lefschetz's hyperplane theorem, $X_0$ has Picard number $1$. We now show that $X_0$ satisfies the conditions of Theorem \ref{Theorem2}.

\begin{lemma}
Let $\zeta$ be a non-zero movable class in $X_0$. Then $\zeta .c_2(X_0)>0$.
\label{LemmaCompleteIntersection}\end{lemma}
\begin{proof}
Let $d_1,\ldots ,d_{N-3}$ be the degrees of $V_1,\ldots ,V_{N-3}$.  Let $h$ be the class of a hyperplane on $X$. The Chern classes of the normal bundle $N_{X_0/\mathbb{P}^n}$ is given by the formula
\begin{eqnarray*}
c(N_{X_0/\mathbb{P}^n})=\prod _{j=1}^{N-3}(1+d_jh).
\end{eqnarray*}
In particular,
\begin{eqnarray*}
c_1(N_{X_0/\mathbb{P}^n})&=&(\sum _{j}d_j)h,\\
c_2(N_{X_0/\mathbb{P}^n})&=&(\sum _{i<j}d_id_j)h^2.
\end{eqnarray*}
From the exact sequence
\begin{eqnarray*}
0\rightarrow T_{X_0}\rightarrow T_{\mathbb{P}^4}|_{X_0}\rightarrow N_{X_0/\mathbb{P}^3}\rightarrow 0,
\end{eqnarray*}
and the splitting principle for Chern classes, it follows that
\begin{eqnarray*}
c_1(X_0)&=&c_1(\mathbb{P}^n)|_{X_0}-c_1(N_{X_0/\mathbb{P}^n})= ((n+1)-\sum _jd_j)h,\\
c_2(X_0)&=&c_2(\mathbb{P}^n)|_{X_0}-c_2(N_{X_0/\mathbb{P}^n})-c_1(X_0)c_1(N_{X_0/\mathbb{P}^n})\\
&=&(\frac{(n+1)n}{2}-\sum _{i<j}d_id_j-(n+1)\sum _{j}d_j+(\sum _{j}d_j)^2)h^2.
\end{eqnarray*}
We have
\begin{eqnarray*}
&&\frac{(n+1)n}{2}-\sum _{i<j}d_id_j-(n+1)\sum _{j}d_j+(\sum _{j}d_j)^2\\
&=&[\frac{n-4}{2(n-3)}(\sum _jd_j)^2-\sum _{i<j}d_id_j]+[\frac{n(n+1)}{2}+\frac{n-2}{2(n-3)}(\sum _jd_j)^2-(n+1)\sum _jd_j].
\end{eqnarray*}
By Cauchy-Schwarz inequality, the first bracket on the right hand side of the above expression is non-negative. We now show that the second bracket is positive. We define $x=\sum _jd_j$. Then $x$ is a positive integer which is $\geq n-3$, and the second bracket is quadratic in $x$:
\begin{eqnarray*}
\frac{n(n+1)}{2}+\frac{n-2}{2(n-3)}(\sum _jd_j)^2-(n+1)\sum _jd_j=\frac{n(n+1)}{2}-(n+1)x+\frac{(n-2)}{2(n-3)}x^2=: g(x).
\end{eqnarray*}
The critical point of $g$ is $x_0=(n+1)(n-3)/(n-2)<n$. Hence, to show that $g(x)>0$ for all positive integer $x\geq n-3$, it suffices to show that $g(n-3),g(n-2),g(n-1),g(n)>0$ for any positive integer $n\geq 4$. We now check this latter claim. 

For $x=n-3$
\begin{eqnarray*}
g(n-3)=6>0.
\end{eqnarray*}
(Note that in this case all $d_j$ are $1$ and $X_0$ is no other than $\mathbb{P}^3$.)

For $x=n-2$, using that $(n-2)^2> (n-1)(n-3)$, we obtain
\begin{eqnarray*}
g(n-2)>\frac{n(n+1)}{2}-(n^2-n-2)+\frac{(n-2)(n-1)}{2}=3>0.
\end{eqnarray*}

For $x=n-1$, we have
\begin{eqnarray*}
g(n-1)=\frac{2(n-2)}{(n-3)}>0.
\end{eqnarray*}

For $x=n$, we have
\begin{eqnarray*}
g(n)=\frac{1}{(n-3)}>0.
\end{eqnarray*}

 A movable class is in particular  psef, i.e. can be represented by a positive closed current. Hence,  if $\zeta$ is a non-zero movable class on $X_0$ then $\xi .c_2(X_0)>0$.  Hence, Theorem \ref{Theorem2} can be applied for such a $X_0$. 

\end{proof}

\subsection{A generalization of Theorem \ref{TheoremResolutionTheMapJ}}
The proof of Theorem \ref{Theorem3} shows that the conclusion is still valid in the following more general setting. Let $\pi _1:X_1\rightarrow X_0=\mathbb{P}^3$ be the blowup at $n$ points $p_1,\ldots ,p_n$. Let $E_1,\ldots ,E_n$ be the exceptional divisors. Let $D_1,\ldots ,D_m\subset X_1$ be pairwise disjoint smooth curves. Let $X=X_2$ be the blowup of $X_1$ at $D_1,\ldots ,D_m$. We define
\begin{eqnarray*}
\gamma :=\sum _j\deg (\pi _1)_*(D_j).
\end{eqnarray*} 
Assume that there is $\lambda >0$ such that for any $l$:
\begin{eqnarray*}
\sum _{j}E_l.D_j\leq \lambda ,
\end{eqnarray*}
and moreover
\begin{eqnarray*}
\frac{6+\gamma }{\lambda}>\frac{11}{2}.
\end{eqnarray*}
Moreover, assume that for any $j$ 
\begin{eqnarray*}
(\frac{1}{2}+\frac{1}{\lambda})c_1(X_1).D_j\geq \frac{g_j-1}{2},
\end{eqnarray*}
where $g_j$ is the genus of $D_j$.

\section{A possible application to the Ueno's threefold}
\label{SectionPossibleApplication}

Let $E_{\sqrt{-1}}$ be an elliptic curve with an automorphism of order $4$, which we denote by $\sqrt{-1}$. In \cite{ueno}, Ueno asked whether the quotient variety $E_{\sqrt{-1}}^3/\sqrt{-1}$ is rational. Campana \cite{campana} showed that the variety is rationally connected. Then, by a combination of the two papers \cite{catanese-oguiso-truong} and \cite{colliot-thelene}, it follows that $E_{sqrt{-1}}^3/\sqrt{-1}$ is rational. Previously, a similar construction, using instead an elliptic curve with an automorphism of order $3$, has been shown to be rational (see \cite{oguiso-truong}). 

The automorphism $\sqrt{-1}$ on $E_{\sqrt{-1}}^3$ has $8$ fixed points and $64-8$ points of period $2$. Therefore, $E_{\sqrt{-1}}^3/\sqrt{-1}$ has $8+28=36$ singular points. Let $X_4$ be the minimal resolution of $E-{\sqrt{-1}}^3/\sqrt{-1}$, that is $X_4$ is the blowup of $E_{\sqrt{-1}}^3/\sqrt{-1}$ at the $36$ singular points.  

Since $X_4$ is birational equivalent to $\mathbb{P}^3$, by the weak factorization theorem, $X_4$ can be obtained from $\mathbb{P}^3$ by  a combination of smooth blowups and blowdowns. It is then natural to ask the following question: 

{\bf Question 2.} Can $X_4$ be obtained from $\mathbb{P}^3$ or $\mathbb{P}^2\times \mathbb{P}^1$ or $\mathbb{P}^1\times \mathbb{P}^1\times \mathbb{P}^1$ by a finite composition of smooth blowups only? 

This question is interesting in several aspects. First, the two dimensional analogue, that is the minimal resolution of $E_{\sqrt{-1}}^2/\sqrt{-1}$, has been shown to be a finite composition of point blowups starting from $\mathbb{P}^1\times \mathbb{P}^1$ in \cite{campana}. Last, the final proof that $X_4$ is rational in \cite{colliot-thelene} is rather abstract. Hence, if the answer to Question 2 is affirmative, it will give an explicit proof that $X_4$ is rational. 

We note that the smooth threefold $X_4$ has automorphisms $f$ coming from the complex torus $E_{\sqrt{-1}}^3$ with $\lambda _1(f)\not= \lambda _2(f)$. Therefore, from the discussion in the introduction of this paper, it is plausible to conclude that the answer to Question 2 is negative. The purpose of this section is to give more weight to this speculation.

We first show that if the answer for Question 2 is affirmative, then centers of the individual blowups  must be smooth rational curves. In the below, for any quasi-projective variety $Z$ we will denote by $\chi (Z)$ the Euler characteristic with compact support. For a smooth projective manifold $Z$, we denote by $\rho (Z)$ the Picard number of $Z$. 

\begin{theorem}
Let $X_0$ be any smooth projective threefold such that $\chi (X_0)=2+2\rho (X_0)$ (for example, $X_0$ is $\mathbb{P}^3$, $\mathbb{P}^2\times \mathbb{P}^1$ or $\mathbb{P}^1\times \mathbb{P}^1\times \mathbb{P}^1$).  Assume that $X_4$ can be obtained from the $X_0$ by a finite composition of smooth blowups. Then the curves which are centers of the blowups must be smooth rational curves. 
\label{TheoremPossibleBlowup}\end{theorem} 
\begin{proof}
We divide the proofs into several steps. 

Step 1. We claim that $\rho (X_4)=45$. In fact, $E_{\sqrt{-1}}^3$ has Picard number $9$. Also, the blowup $X_4\rightarrow E_{\sqrt{-1}}^3/\sqrt{-1}$ has $36$ exceptional divisors, one for each singular points. Hence the Picard number of $X_4$ is $9+36=45$. 

Step 2. We claim that $\chi (X_4)=92$. In fact, first we consider the quotient map $\sigma :E_{\sqrt{-1}}^3\rightarrow E_{\sqrt{-1}}^3/\sqrt{-1}$. Let $A\subset E_{\sqrt{-1}}^3$ be the set of fixed points of $\sqrt{-1}$, and $B\subset E_{\sqrt{-1}}^3$ the set of points of period $2$ of $\sqrt{-1}$. As mentioned before, the cardinals of $|A|$ and $|B|$ are $8$ and $56$, and the cardinals of $\sigma (A)$ and $\sigma (B)$ are $8$ and $28$. Since the map $\pi :E_{\sqrt{-1}}^3-(A\cup B)\rightarrow E_{\sqrt{-1}}^3/\sqrt{-1}-(\sigma (A)\cup \sigma (B))$ is a $4:1$ map, by the excision property we get
\begin{eqnarray*}
\chi (E_{\sqrt{-1}}^3/\sqrt{-1}-(\sigma (A)\cup \sigma (B)))&=&\chi (E_{\sqrt{-1}}^3-(A\cup B))/4\\
&=&(\chi (E_{\sqrt{-1}}^3)-\chi (A\cup B))/4\\
&=&(0-64)/4=-16.\end{eqnarray*}   
Hence, by the excision property, we have $\chi (E^3_{\sqrt{-1}}/\sqrt{-1})=-16 +36=20$. 

Next, we consider the blowup $\pi :X_4\rightarrow E^3_{\sqrt{-1}}/\sqrt{-1}$. This map has $36$ exceptional divisors, each is a $\mathbb{P}^2$. Since $\chi (\mathbb{P}^2)=3$ and the blowup map is $1:1$ outside exceptional divisors, arguing as above we obtain 
\begin{eqnarray*}
\chi (X_4)=(\chi (E^3_{\sqrt{-1}}/\sqrt{-1})-36 \chi (pt))+36\times \chi (\mathbb{P}^2)=20-36+36\times 3=92. 
\end{eqnarray*}
Here $pt$ denotes a point. 

Step 3: $\chi (X_4)=2+2\rho (X_4)$. This follows from Steps 1 and 2.

Step 3. Let $Z$ be a smooth projective threefold and $\pi :Z_1\rightarrow Z$ a point blowup. Then $\chi (Z_1)=\chi (Z)+2$. This follows easily from the properties of the blowup of a threefold.

Step 4. Let $Z$ be a smooth projective threefold and $\pi :Z_1\rightarrow Z$ a blowup at a smooth curve $C\subset Z$. Then $\chi (Z_1)\leq \chi (Z)+2$, with equality if and only if $C$ is a smooth rational curve. Again, this follows easily from the properties of the blowup of a threefold, and the fact that if $C$ is not a smooth rational curve then $h^{1}(C)>0$. 

Step 5: final step. If $X_4$ is a finite composition of smooth blowups of $X_0$, then the number of blowups needed is $\rho (X_4)-\rho (X_0)=45-\rho (X_0)$. From the previous steps we have
\begin{eqnarray*}
92=\chi (X_4)\leq \chi (X_0)+2 (45-\rho (X_0))=2+2\rho (X_0)+90-2\rho (X_0)=92.
\end{eqnarray*} 
Since equality occurs, it follows from Steps 3 and 4 that the centers of the individual blowups must be either a point or a smooth rational curve. 
\end{proof}

Now we show how Theorem \ref{TheoremPossibleBlowup} and Theorems \ref{Theorem3},  \ref{Theorem4} and \ref{Theorem5} almost give the proof that the answer to Question 2 is negative. In fact, let $X_0$ be $\mathbb{P}^3$, $\mathbb{P}^2\times \mathbb{P}^1$ or $\mathbb{P}^1\times \mathbb{P}^1\times \mathbb{P}^1$. Then, $X_0$ satisfies Condition B, while $X_4$ does not satisfy Condition B. Assume that $X_4$ is a finite composition of smooth blowups starting from $X_0$. Let $\pi _j:~Z_{j+1}\rightarrow Z_j$ be an individual blowup in the sequence, where $Z_{j}$ satisfies Condition B. If $\pi _j$ is a point blowup then by Theorem \ref{Theorem3}, $Z_{j+1}$ also satisfies Condition B. If $\pi _j$ is the blowup of a smooth curve $C\subset Z_j$, then by Theorem \ref{TheoremPossibleBlowup}, we have that $C$ must be a smooth rational curve. If $c_1(Z_j).C\not= 2g-2=-2$, then by Theorem \ref{Theorem5} we have that $Z_{j+1}$ also satisfies Condition B. The remaining case is when $c_1(Z_j).C=-2$. But in this case, half of the conditions of part 2 of Theorem \ref{Theorem4} is satisfied. The only condition that is missing is the condition that $C$ is not the only effective curve in its cohomology class. Using part 1 of Theorem \ref{Theorem3}, we can also show that if the normal vector bundle $N_{C/Z_j}$ is not isomorphic to $\mathcal{O}(-2)\oplus \mathcal{O}(-2)$, then $Z_{j+1}$ also satisfies Condition B.

\end{document}